\documentclass[letterpaper, 10 pt, conference]{ieeeconf}  

\IEEEoverridecommandlockouts                              

\usepackage{graphicx}
\usepackage{amsmath}
\usepackage{dsfont}
\usepackage{xcolor}
\usepackage{amsfonts}
\usepackage{amssymb}
\usepackage{subfigure}
\usepackage{tikz}
\usepackage{tkz-berge}
\usetikzlibrary{calc}
\usetikzlibrary{shapes,fit}
\usepackage{verbatim}
\usepackage{setspace}
\usetikzlibrary{shapes,arrows}
\DeclareMathOperator{\im}{im}

\DeclareMathOperator{\spa}{span}
\DeclareMathOperator{\sat}{sat}

\newtheorem{theorem}{Theorem}
\newtheorem{lemma}[theorem]{Lemma}

\newtheorem{definition}[theorem]{Definition}

\newtheorem{corollary}[theorem]{Corollary}
\newtheorem{exmp}{Example}[section]
\newtheorem{remark}[theorem]{Remark}

\title{\LARGE \bf
A graphic condition for the stability of dynamical distribution networks with flow constraints
}

\author{J. Wei$^{1}$ and A.J. van der Schaft$^{2}$
\thanks{*The work of the first author is supported by Chinese Scholarship Council.The research of the second author leading to these
results has received funding from the EU 7th Framework Programme [FP7/2007-2013] under grant agreement no. 257462 HYCON2 Network of
Excellence.}
\thanks{$^{1,2}$J.Wei and A.J. van der Schaft are with Johann Bernoulli Institute for Mathematics and Computer Science, University of Groningen, 9700 AK, the Netherlands.
        {\tt\small J.Wei@rug.nl }
        {\tt\small A.J.van.der.Schaft@rug.nl}}%
}

\begin{document}

\maketitle
\thispagestyle{empty}
\pagestyle{empty}

\begin{abstract}
We consider a basic model of a dynamical distribution network, modeled as a directed
graph with storage variables corresponding to every vertex and flow inputs
corresponding to every edge, subject to unknown but constant
inflows and outflows. In \cite{wei2013} we showed how a distributed proportional-integral controller structure, associating with every edge of the graph a controller state, regulates the state
variables of the vertices, irrespective of the unknown constant inflows and
outflows, in the sense that the storage variables converge to the same value
(load balancing or consensus). In many practical cases, the flows on the edges are constrained. The main result of \cite{wei2013} is a sufficient and necessary condition, which only depend on the structure of the network, for load balancing for arbitrary constraint intervals of which the intersection has nonempty interior.
In this paper, we will consider the question about how to decide the steady states of the same model as in \cite{wei2013} with given network structure and constraint intervals.  We will derive a graphic condition, which is sufficient and necessary, for load balancing. This will be proved by a Lyapunov function and the analysis the kernel of incidence matrix of the network. Furthermore, we will show that by modified PI controller, the storage variable on the nodes can be driven to an arbitrary point of admissible set. 
\end{abstract}

\section{INTRODUCTION}
Production-distribution systems form a very important class of systems that
have a large number of applications. In this paper, we pursue an approach
similar to that proposed in \cite{wei2013}. Given the network which is
depicted as a directed graph, we assign a state variable with every vertex of
the graph, and a control input corresponding to flow with every edge, which is
constrained in a given closed interval. Furthermore, the system is open to
environments by some ports, namely some of the vertices serves as terminals,
where an unknown-but-constant flow may enter or leave the network in such a way
that the total sum of inflows and outflows is equal to zero. 

There are many relevant references on this topic. In \cite{depersis}, a class
of cooperative control algorithms is proposed in the context of distribution
network under time-varying exogenous in/outflows. The author dealt with
constraint for control input. However, the constraint intervals are all
symmetric with respect to the origin. In \cite{Blanchini00}, the main
problem is the joint presence of buffer/flow capacity and of the unknown
in/outflows. A discontinuous control strategy is proposed to drive the system
states, whose components are also storage in nodes, into consensus for all
possible unknown in/outflows by using control input that are subject to hard
bounds. In \cite{wei2013}, a sufficient and necessary condition is derived that
with PI controller and arbitrary constraint intervals whose intersection has
nonempty interior, the state variables corresponding to all vertices converge to
consensus if and only if the directed graph is strongly connected and balanced. In \cite{Ren08,jaya1}, a similar model with constraint is considered from physical perspective. 

The control problem to be studied here is to derive a criteria, which only
depends on the structure of network and flow constraints, to decide whether or
not a distributed control structure (the control input corresponding to each
edge only depending on the
difference of state variables of its two endpoints) will make the state
variables associated to all vertices converge to the same value equal to the
average of the initial condition for given constraint intervals and constant
unknown in/outflows. Notice that this control strategy is decentralized.

The organization of this paper is as follows. Some preliminaries and notations
will be given in Section 2. In Section 3, the class of systems and a basic
assumption about flow constraints under study will be introduced. 

The main contribution of this paper resides in Section 4, 5. In Section
4, it will be shown that the state variables associated to all the vertices 
converge to consensus, if and only if the network and flow constraints satisfy
IPC (interior point condition) which will be defined later. Similar to
\cite{Blanchini00}, if the information of desirable state is available, we can
modify the PI controller such that the storage converge to the desirable
point. This will be explained in Section 5. Finally, Section 6 contains the
conclusion.

\section{Preliminaries and notations}

First we recall some standard definitions regarding directed graphs,
as can be found e.g. in \cite{Bollobas98}. A \textit{directed graph}
$\mathcal{G}$ consists of a finite set $\mathcal{V}$ of \textit{vertices}
and a finite set $\mathcal{E}$ of \textit{edges}, together
with a mapping from $\mathcal{E}$ to the set of ordered pairs of
$\mathcal{V}$, where no self-loops are allowed. Thus to any edge
$e\in\mathcal{E}$ there corresponds an ordered pair
$(v,w)\in\mathcal{V}\times\mathcal{V}$
(with $v\not=w$), representing the tail vertex $v$ and the head
vertex $w$ of this edge. An undirected graph $\mathcal{G}^o$ is obtained from
$\mathcal{G}$ by ignoring the orientation of the edges. A cycle in
$\mathcal{G}^o$ is a closed path in which the internal vertices are distinct. An
oriented cycle in $\mathcal{G}$ is a cycle in $\mathcal{G}^o$ with an
orientation assigned by an ordering of the vertices in the cycle. Given an
oriented cycle $\mathcal{C}$, we define the vectorial representation of the
cycle $\mathcal{C}$ as $C$ whose component is given as 
\begin{equation}
C_i  = \left\{ \begin{array}{ll}
0 & \textrm{$e_i\notin\mathcal{C}$}\\
1 & \textrm{$e_i\in\mathcal{C}$ and the orientations agree}\\
-1 & \textrm{$e_i\in\mathcal{C}$ and the orientations disagree},
\end{array}
 \right.\\
\end{equation}
A oriented cycle of a directed graph $\mathcal{G}$ which has an orientation
which agrees with the orientations in the graph is called positive circuit.

A directed graph is completely specified by its \textit{incidence
matrix} $B$, which is an $n\times m$ matrix, $n$ being the
number of vertices and $m$ being the number of edges, with $(i,j)^{\text{th}}$
element equal to $1$ if the $j^{\text{th}}$ edge is towards vertex
$i$, and equal to $-1$ if the $j^{\text{th}}$ edge is originating from
vertex $i$, and $0$ otherwise. 
A directed graph is {\it strongly connected} if it is
possible to reach any vertex starting from any other vertex by traversing edges 
following their directions. A directed graph $\mathcal{G}$ is called {\it weakly
connected} if $\mathcal{G}^o$ is connected. A digraph is weakly connected if
and only if $\ker B^T = \spa \mathds{1}_n$. Here $\mathds{1}_n$ denotes the
$n$-dimensional vector with all elements equal to $1$. We omit the subscript if
the dimension of the vector is unambiguous from the context. A digraph that is
not weakly connected falls apart into a number of weakly connected subgraphs,
called the weakly connected components. The number of weakly connected
components is equal to $\dim \ker B^T$. A subgraph
$\mathcal{T}\subseteq\mathcal{G}^o$ is a tree if it is connected and
acyclic, and a spanning tree if it is a tree and contains all the vertices of
$\mathcal{G}^o$.

Given a graph, we define its \textit{vertex space} as the vector space of all
functions from $\mathcal{V}$ to some linear space $\mathcal{R}$. In the rest of
this paper we will take for simplicity $\mathcal{R}=\mathbb{R}$, in which case
the vertex space can be identified with $\mathbb{R}^{n}$. Similarly, we define
its \textit{edge space} as the
vector space of all functions from $\mathcal{E}$ to $\mathcal{R} = \mathbb{R}$,
which can be identified with $\mathbb{R}^{m}$. In this way, the incidence matrix
$B$ of the graph can be also regarded as the matrix representation of a linear
map from the edge space $\mathbb{R}^m$ to the vertex space $\mathbb{R}^n$.

A cone in $\mathbb{R}^n$ is a closed subset $K$ such that $K\cap\{-K\}=\{0\}$
and $\alpha K+\beta K \subseteq K$ for all $\alpha, \beta \geq 0$. A cone is
generated by a set of vectors in $K$ if any $x\in K$ can be written as a linear
combination of vectors in the set, using only nonnegative coefficients. The
dimension of $K$ is the number of elements in a minimal generating set.

\noindent
{\bf Notation}: For $a,b\in\mathbb{R}^m$ the notation $a \leqslant b$ will
denote element-wise inequality $a_i \leq b_i,\,i=1,\ldots,m$. For $a_i \leq
b_i,\,i=1,\ldots,m$ the multidimensional
saturation function
$\sat(x\,;a,b): \mathbb{R}^m\rightarrow\mathbb{R}^m$ is defined as
\begin{equation}
\sat(x\,;a,b)_i  = \left\{ \begin{array}{ll}
a_i & \textrm{if $x_i< a_i,$}\\
x_i & \textrm{if $a_i\leq x_i\leq b_i,$}\\
b_i & \textrm{if $x_i> b_i$},
\end{array}
\, i=1,\ldots,m. \right.
\end{equation}
Its integral $S(x\,;a,b):\mathbb{R}^m\rightarrow\mathbb{R}^m$ is defined as 
\begin{equation}
 S(x\,;a,b)_i = \int_{0}^{x_i}\sat(y\,;a_i,b_i)dy.
\end{equation}

\section{A dynamic network model with input constraints}

Let us consider the following dynamical system defined on the graph (\cite{schaftNECSYS10,schaftCDC08,schaftSIAM})
\begin{equation}\label{system1}
\begin{array}{rcl}
\dot{x} & = & Bu + Ed, \quad x \in \mathbb{R}^n, u \in \mathbb{R}^m, \quad d \in
\mathbb{R}^k \\[2mm]
y & = & B^T \frac{\partial H}{\partial x}(x), \quad y \in \mathbb{R}^m,
\end{array}
\end{equation}
where $H: \mathbb{R}^n \to \mathbb{R}$ is a differentiable function, and
$\frac{\partial H}{\partial
x}(x)$ denotes the column vector of partial derivatives of $H$. Here the
$i^{\text{th}}$
element $x_i$ of the state vector $x$ is the state variable
associated to the $i^{\text{th}}$ vertex, while $u_j$ is a flow input variable
associated to the
$j^{\text{th}}$ edge of the graph. And $E$ is an $n \times k$ matrix
whose columns consist of exactly one entry equal
to $1$ (inflow) or $-1$ (outflow), while the rest of the elements is zero. Thus
$E$ specifies the $k$ terminal vertices where flows can enter or leave the
network (\cite{vanderschaftmaschkearchive}).
System (\ref{system1}) defines a
port-Hamiltonian system (\cite{vanderschaftmaschkearchive,
vanderschaftbook}), satisfying the energy-balance
\begin{equation}
\frac{d}{dt}H = u^Ty+\frac{\partial^T H}{\partial x}(x)Ed.
\end{equation}

As explained in \cite{wei2013}, when $d\neq0$, the proportional control will
not be sufficient to reach load balancing. Hence we consider a
proportional-integral (PI) controller given by the dynamic output feedback
\begin{equation}\label{unconstrained-controller}
 \begin{aligned}
  \dot{x}_c & = RB^T\frac{\partial H}{\partial x}(x)\\
u & = -RB^T\frac{\partial H}{\partial x}(x)-R\frac{\partial H_c}{\partial
x_c}(x_c)
 \end{aligned}
\end{equation}
Then the closed-loop can be represented as the following
port-Hamiltonian system
\begin{equation}\label{closedloop}
 \begin{bmatrix} \dot{x} \\[2mm] \dot{x}_c \end{bmatrix} =
\begin{bmatrix} -BRB^T & -BR \\[2mm] RB^T & 0 \end{bmatrix}
\begin{bmatrix} \frac{\partial H}{\partial x}(x) \\[2mm] \frac{\partial
H_c}{\partial x_c}(x_c) \end{bmatrix} +
\begin{bmatrix} E \\[2mm] 0 \end{bmatrix} d,
\end{equation}
with $H_{tot}(x,x_c)=H(x)+H_c(x_c)$.

Since $\mathds{1}^T\dot{x}=\mathds{1}^TEd$, the system has a steady state if
and only if $\mathds{1}^TEd=0$. For any weakly connected graph with
$n$ vertices, $Ed\in\im B$, for all $Ed$ such that $\mathds{1}^TEd=0$. Suppose
now
the constant disturbance $\bar{d}$ and satisfies the
{\it matching condition}, i.e. there exists a 
controller state $\bar{x}_c$ such that
\begin{equation}\label{matching}
E \bar{d} = B\frac{\partial H_c}{\partial x_c}(\bar{x}_c).
\end{equation}


In many practical cases, the elements of the vector of flow inputs
$u \in
\mathbb{R}^m$
corresponding to the edges of the graph will be {\it constrained}, that is
\begin{equation}
 u \in\mathcal{U}:=\{u\in\mathbb{R}^m\mid u^-\leqslant u\leqslant u^+\}
\end{equation}
for certain vectors $u^-$ and $ u^+$ satisfying $u^-_i<
u^+_i, i=1,\ldots,m$. In our previous paper \cite{wei2013} we focused on
the cases where $u^-_i\leqslant0<u^+_i,i=1,2,\ldots,m.$ or $0\leq u^-_i<u^+_i,i=1,2,\ldots,m$ of which the intersection has nonempty interior. In the present paper we
consider {\it arbitrary}
constraint intervals, necessitating a novel approach to the problem.

Thus we consider a general constrained version of the PI controller given as
\begin{equation}\label{PIconstrained}
\begin{array}{rcl}
\dot{x}_c & = & Ry ,\\[2mm]
u & = &\sat\big(-Ry - R\frac{\partial H_c}{\partial x_c}(x_c)\,;u^-,u^+\big)
\end{array}
\end{equation}
For simplicity of exposition we consider throughout the rest of this paper the
identity gain matrix $R=I$.  Furthermore we throughout assume that
the Hessian matrix of Hamiltonian $H(x)$ is positive definite for any $x$ and
we only consider $H_c(x_c) = \frac{1}{2} \| x_c \|^2$.
Then the system (\ref{system1}) with constraint PI controller
(\ref{PIconstrained}) is given as
\begin{equation}\label{closedloop-sat-disturb}
\begin{aligned}
\dot{x} & =  B\sat\big(-B^T\frac{\partial H}{\partial
x}(x)-x_c\,;u^-,u^+\big)+E\bar{d},
\\[2mm]
\dot{x}_c & =  B^T\frac{\partial H}{\partial x}(x),
\end{aligned}
\end{equation}

\begin{remark}
For arbitrary diagonal positive definitive gain matrix $R$, we can use as
Lyapunov function instead of (\ref{lyapunov}) the expression 
\begin{equation}
 V(x,x_c)=\mathds{1}^{T}R^{-1} S\big(-RB^{T}\frac{\partial H}{\partial
x}(x)-Rx_c\,;u^-,u^+\big)+H(x)
\end{equation}
 as Lyapunov function and obtain the same conclusions.
\end{remark}

The constrained system is different from the one in (\cite{Ren08}) where the saturation is added separately on the proportional and integral part of the controller.

In the rest of this section, we will first show how the disturbance can be {\it
absorbed} into the constraint intervals. Indeed, for
any $\eta\in\mathbb{R}^n$, we have the identity
\begin{equation}\label{identity}
\sat(x-\eta\,;u^-,u^+)+\eta=\sat(x\,;u^-+\eta,u^++\eta).
\end{equation}
Therefore for an in/out flow $\bar{d}$ satisfying the matching
condition, i.e., such that there exists $\bar{x}_c$ with $B\bar{x}_c=E\bar{d}$,
we can rewrite system (\ref{closedloop-sat-disturb}) as
\begin{equation}\label{disturbance_in_const}
 \begin{aligned}
  \dot{x} & = B\sat(-B^T\frac{\partial H}{\partial
x}(x)-x'_c\,;u^-+\bar{x}_c,u^++\bar{x}_c), \\
\dot{x}'_c & = B^T\frac{\partial H}{\partial
x}(x),
 \end{aligned}
\end{equation}
where $x'_c=x_c-\bar{x}_c$.
It follows that, without loss of generality, we can restrict ourselves to the
study of the closed-loop system
\begin{equation}\label{closedloop-sat}
\begin{aligned}
\dot{x} & =  B\sat\big(-B^T\frac{\partial H}{\partial
x}(x)-x_c\,;u^-,u^+\big),
\\[2mm]
\dot{x}_c & =  B^T\frac{\partial H}{\partial x}(x).
\end{aligned}
\end{equation}
for general $u^-$ and $ u^+$ with $u^-_i\leq u^+_i, i=1,\ldots,m$.

Next, we will show how the orientation can be made compatible with the flow
constraints.

Any {\it bi-directional} edge whose constraint interval satisfies
$u^-_i<0<u^+_i$, it can be divided into {\it two uni-directional} edges with
constraint intervals $[u^-_i,0], [0,u^+_i]$ respectively, and the same
orientation. This follows from 
\begin{equation}
 \sat(u_i;u^-_i,u^+_i)=\sat(u_i;u^-_i,0)+\sat(u_i;0,u^+_i)
\end{equation}
for any $u^-_i<0<u^+_i$.

Furthermore, we may change the {\it orientation} of
some of the edges of the graph at will; replacing the corresponding columns
$b_i$ of the
incidence matrix $B$ by $-b_i$. By the identity
\begin{equation}\label{identity1}
\sat(-x\,;u_i^-,u_i^+)=-\sat(x\,;-u_i^+,-u_i^-),
\end{equation}
we may therefore assume {\it without loss of generality} that the orientation of
the graph
is
chosen such that
\begin{equation}
u^+_i>0,\, i=1,2,\ldots,m.
\end{equation}

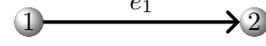
\begin{figure}[ht]
\begin{center}
\begin{tikzpicture}
\tikzstyle{EdgeStyle}    = [thin,double= black,
                            double distance = 0.5pt]
\useasboundingbox (0,0) rectangle (4cm,0.5cm);
\tikzstyle{VertexStyle} = [shading         = ball,
                           ball color      = white!100!white,
                           minimum size = 20pt,%
                           inner sep       = 1pt,]
\Vertex[style={minimum
size=0.2cm,shape=circle},LabelOut=false,L=\hbox{$1$},x=1cm,y=0.2cm]{v1}
\Vertex[style={minimum
size=0.2cm,shape=circle},LabelOut=false,L=\hbox{$2$},x=4cm,y=0.2cm]{v2}
\draw
(v1) edge[->,>=angle 90,thin,double= black,double distance = 0.5pt]
node[above]{$e_1$} (v2);
\end{tikzpicture}
\caption{Illustrative graph}\label{figure_ex1}
\end{center}
\end{figure}

\begin{exmp}
Consider the graph in Fig.\ref{figure_ex1}, where the constraint
interval for edge $e_1$ is $[-2,-1]$. The network
is equivalent to the network where the edge direction is reversed from $v_2$ to
$v_1$ 
while the constraint interval is modified into $[1,2]$.
\end{exmp}
By dividing bi-directional edges into uni-directional ones and changing
orientations afterwards, we can therefore assume that
\begin{equation}\label{assumption1}
 u^+_i\geq u^-_i \geq0, \,\, i=1,2,\ldots,m.
\end{equation}
where the two equality signs do not hold at the same time.

Assumption (\ref{assumption1}) will be standing throughout the rest of the
paper. In general, we will say that orientation of the the graph is {\it
compatible with the flow
constraints} if (\ref{assumption1}) hold.

\section{Convergence conditions for the closed-loop dynamics with general flow
constraints}

The following lemma is a cornerstone of this paper.

\begin{lemma}( \cite{gatermann2005},Lemma 3.2.9 in
\cite{gatermann2002},\cite{othmer1981})\label{positive circuit}
 The set of minimal generators of the convex polyhedral cone
$\ker B\cap\mathbb{R}_{\geq0}^m$ is composed of positive circuits. 
\end{lemma}

\begin{definition}\label{interior point}
 (Interior Point Condition) Given a directed graph with arbitrary constraints
$[u^-, u^+]$ (maybe not compatible with the orientation), the network will be
said to satisfy the
interior point condition
if there exists a vector $z\in[u^-, u^+]$ such that 
\begin{equation}
 B\sat(z;u^-, u^+)=Bz=0,
\end{equation}
and the set of edges along which the corresponding element of $z$ is an interior
point of the constraint interval contains a spanning tree.
\end{definition}

\begin{remark}
 We will show that interior point condition is independent of the choice of
$\beta\in\ker B$, i.e. a graph $\mathcal{G}$ with constraints $[u^-,u^+]$
satisfies the interior point condition, then $\mathcal{G}$ with constraints
$[u^-+\beta,u^++\beta]$ also satisfies it for any $\beta\in\ker B$. This
problem can be caused in $(\ref{closedloop-sat})$, since for any $\beta\in\ker
B$, system $(\ref{closedloop-sat})$ can be rewritten with new constraints
$[u^-+\beta,u^++\beta]$ and new edge states $x'_c=x_c-\beta$. Indeed,
$z\in[u^-,u^+] \Longleftrightarrow z+\beta\in[u^-+\beta,u^++\beta]$ and
$z_i\in int[u^-_i,u^+_i] \Longleftrightarrow
z_i+\beta_i\in int[u^-_i+\beta_i,u^+_i+\beta_i]$. That verifies our statement.
\end{remark}

\begin{remark}
 For any network with compatible orientation, we can
assume in Definition \ref{interior point} that $z\in\mathbb{R}^{m}_{\geq0}$.
This will be assumed throughout the rest of this paper.
\end{remark}

Basing on the vector $z$ from the interior point condition, we can divide the
edge set $\mathcal{E}$ into the following four subsets
\begin{equation}
 \begin{aligned}
  \mathcal{E}_0(z; u^-, u^+) & = \{ e_i \mid e_i\in\mathcal{E}, z_i=0\} \\
\mathcal{E}_1(z; u^-, u^+) & = \{ e_i \mid e_i\in\mathcal{E}, z_i>0\}\\
\mathcal{E}_2(z; u^-, u^+) & = \{ e_i \mid e_i\in\mathcal{E}, z_i\in int
[u^-_i,u^+_i]\}\\
\mathcal{E}_3(z; u^-, u^+) & = \{ e_i \mid e_i\in\mathcal{E}, z_i=u^-_i \textrm{
or }
z_i=u^+_i\}
 \end{aligned}
\end{equation}
with
$\mathcal{E}_0\cup\mathcal{E}_1=\mathcal{E}_2\cup\mathcal{E}
_3=\mathcal{E}$. The subgraph
$\mathcal{G}_i=\{\mathcal{V},\mathcal{E}_i\},i=0,1,2,3,$ can be defined
respectively.

\begin{lemma}\label{weakly-strongly connected}corollary
Let $\mathcal{G}$ be a weakly connected directed graph with compatible
constraint intervals
$[u^-, u^+]$. Then $\mathcal{G}$ is strongly
connected if it satisfies the interior point condition.
\end{lemma}

\begin{proof}
 Let $z\in[u^-, u^+]\cap\mathbb{R}^{m}_{\geq0}$ be such that $Bz=0$. 
Then by Lemma \ref{positive circuit}, $z$ can be represented as a positive
linear combination of positive circuits. Furthermore, the set of edges along
which $z$ is a interior point of the constraint intervals contains a spanning
tree, then $\mathcal{G}$
contains a strongly connected subgraph
$\mathcal{G}_1=\{\mathcal{V},\mathcal{E}_1(z;u^-, u^+)\}$. In conclusion,
$\mathcal{G}$ is strongly connected.
\end{proof}

If $\mathcal{G}$ is strongly connected it must contain positive
cycles, and it is easy to see that in this case any cycle that is not positive
can be written as linear combination of positive circuits. Consequently,
when $\mathcal{G}$ is strongly connected, positive circuits compose a basis
of $\ker{B}$.

\medskip

Let $z\in[u^-, u^+]\cap\mathbb{R}^m_{\geq0}$ be the vector from the interior
point condition. By Lemma \ref{positive circuit}, $z$ can be represented as
\begin{equation}\label{z representation}
 z=\sum_{i=1}^{k}\alpha_i C_i\,\, \alpha_i>0
\end{equation}
where $\mathcal{C}_i$ is a positive circuit of $\mathcal{G}_1$ and
$C_i$ is vectorial representation of $\mathcal{C}_i,i=1,\ldots,k$. Denote the set of
these $k$ positive circuits as
$\tilde{\mathcal{C}}=\{\mathcal{C}_1,\mathcal{C}_2,
\cdots,\mathcal{C}_k\}$. By Lemma
\ref{weakly-strongly connected}, the graph $\mathcal{G}_1(z;u^-, u^+)\}$ can be
covered by $\tilde{\mathcal{C}}$.

Next we will explain the relation between $\tilde{\mathcal{C}}$ and $[u^-,u^+]$.
Suppose an edge $e_h$ is not overlapped in $\tilde{\mathcal{C}}$, i.e. there is
only one positive circuit in $\mathcal{C}_i\in\tilde{\mathcal{C}}$ such that
$e_h\in\mathcal{C}_i$, then clearly 
\begin{equation}
 \alpha_i\in[u^-_h,u^+_h]
\end{equation}
 where $\alpha_i$ is given as in (\ref{z representation}). However, if an edge
$e_h$ is overlapped in $\tilde{\mathcal{C}}$, without loss of generality, say
$e_h$ belongs to $\mathcal{C}_1,\mathcal{C}_2,\cdots,\mathcal{C}_d$, then we
have 
\begin{equation}
 \sum_{i=1}^{d}\alpha_i \in[u^-_h,u^+_h].
\end{equation}

\begin{lemma}\label{lower bound}
 Consider the dynamical system (\ref{closedloop-sat}) defined on the network
satisfying the interior point condition. Then 
\begin{spacing}{1.5}
(i) Along every trajectory $(x(t),x_c(t)),
t\geqslant 0,$ of (\ref{closedloop-sat}), the function 
\begin{equation}\label{lyapunov}
\begin{aligned}
V(x(t),x_c(t))&=\mathds{1}^{T} S\big(-B^{T}\frac{\partial H}{\partial
x}(x(t))-x_c(t)\,;u^-,u^+\big) \\
&+H(x(t))
\end{aligned}
\end{equation}
is bounded from below,

(ii) The trajectory $(x(t),x_c(t)), t\geqslant0,$ is bounded,

(iii) $\lim_{t\rightarrow\infty}\dot{V}(x(t),x_c(t))=0$,
\end{spacing}
\end{lemma}

\begin{proof}
(i)
Since $H(x)$ is positive definite, we only need to show that the components of
$V_1(x,x_c):= S\big(-B^{T}\frac{\partial H}{\partial x}(x)-x_c\,;u^-,u^+\big)$
are bounded from below. Suppose the $i$th component of $V_1(x(t),x_c(t))$
converges
to $-\infty$, by the property of integral of saturation function, this holds if
and only if on the $i$th edge $\big(-B^{T}\frac{\partial H}{\partial
x}(x(t))-x_c(t)\big)_i\rightarrow -\infty$ and $u^{-}_{i}>0$. 

Let us define two subsets of $\mathcal{E}$
\begin{equation}
 \begin{aligned}
  \mathcal{E}_{-\infty} & = \{ e_i\in\mathcal{E} \mid \big(-B^{T}\frac{\partial
H}{\partial
x}(x(t))-x_c(t)\big)_i\rightarrow -\infty\}, \\
\mathcal{E}_{+\infty} & = \{ e_i\in\mathcal{E} \mid \big(-B^{T}\frac{\partial
H}{\partial
x}(x(t))-x_c(t)\big)_i\rightarrow +\infty\}.
 \end{aligned}
\end{equation}

When $\mathcal{E}_{-\infty}=\emptyset$, it is easy to see that $V(x(t),x_c(t))$
is bounded from below.

When $\mathcal{E}_{-\infty}\neq\emptyset$, for large enough $t$, on the edges
of $\mathcal{E}_{-\infty}$, the summation of corresponding components of
$V_1(x(t),x_c(t))$ is equal to
\begin{equation}
 \sum_{e_i\in\mathcal{E}_{-\infty}} u^-_i\big( -B^T\frac{\partial
H}{\partial x}(x(t))-x_c(t)\big)_i.
\end{equation}
up to constants which depend on the initial condition.
Similarly, on the edges of $\mathcal{E}_{+\infty}$, the summation is equal to 
\begin{equation}
 \sum_{e_i\in\mathcal{E}_{+\infty}} u^+_i\big( -B^T\frac{\partial
H}{\partial x}(x(t))-x_c(t)\big)_i.
\end{equation}

Let $z\in\ker B\cap\mathbb{R}^m_{\geq0}$ be the vector from the interior
point condition, we have 
\begin{equation}\label{sum_inputs}
 \begin{aligned}
  &0=z^TB^T\frac{\partial H}{\partial x}(x)& \\[2mm]
\Rightarrow &\sum_{e_i\in\mathcal{E}}z_ix_{c_{i}}(t)
=\sum_{e_i\in\mathcal{E}}z_ix_{c_{i}}(0)\\[2mm]
\Rightarrow &\sum_{e_i\in\mathcal{E}}z_i\big( -B^T\frac{\partial
H}{\partial x}(x(t))-x_c(t)\big)_i
=-\sum_{e_i\in\mathcal{E}}z_ix_{c_{i}}(0), \\
& \forall t>0
 \end{aligned}
\end{equation}

Then by (\ref{sum_inputs}) and the fact that $u^-_i\leq z_i\leq u^+_i, \forall
e_i\in\mathcal{E}$, we have that function (\ref{lyapunov}) is bounded from
below.

(ii) Notice that $\dot{V}=-\dot{x}^T\frac{\partial^2 H}{\partial
x^2}\dot{x}\leq0$. 

Suppose that $x(t), t\geq0$, is not bounded, then there exists a sequence
$\{t_k\}, t_k\geq0$ such that 
\begin{equation}
 \lim_{k\rightarrow\infty}\|x(t_k)\| = +\infty.
\end{equation}
Since $H(x)$ is unbounded, this implies 
\begin{equation}
 \lim_{k\rightarrow\infty} V(x(t_k),x_c(t_k))=+\infty.
\end{equation}
This is a contradiction with $\dot{V}\leq0.$

Suppose $x_c$ is unbounded, we first show that this can not happen only on the
edges of $\mathcal{E}_3(z;u^-,u^+)$. Indeed, suppose $x_c$ is
unbounded on $e_i\in\mathcal{E}_3(z;u^-,u^+)$ where $e_i\sim(x_p,x_q),$
then by the property of dynamics of $x_c$ in (\ref{closedloop-sat}), along any
positive path from $x_q$ to $x_p$, there exists at least one edge, on which
$x_c$ is unbounded, belongs to $\mathcal{E}_3(z;u^-,u^+)$. Then $\{\mathcal{V},
\mathcal{E}\setminus\mathcal{E}_3(z;u^-,u^+)\}=\{\mathcal{V},
\mathcal{E}_2(z;u^-,u^+)\}$ is not weakly connected which is a contradiction
with the fact that $\mathcal{E}_2(z;u^-,u^+)$ contains a spanning tree.

Then if $x_c$ is unbounded, there must be some edges of
$\mathcal{E}_2(z;u^-,u^+)$ on which $x_c$ is unbounded. However, if on
$e_i\in\mathcal{E}_2(z;u^-,u^+)$, there exists a
sequence $\{t_k\}, t_k\geq0$ such that
\begin{equation}
 \lim_{k\rightarrow\infty}\|x_{c_i}(t_k)\| = +\infty,
\end{equation}
then similar to (i), by using (\ref{sum_inputs}) and the fact that
$u^-_i<z_i<u^+_i, \forall e_i\in\mathcal{E}_2(z;u^-,u^+)$, we have  
\begin{equation}
 \lim_{k\rightarrow\infty} V(x(t_k),x_c(t_k))=+\infty.
\end{equation}
This is a contradiction to $\dot{V}\leq0$ again.

In conclusion, $(x,x_c)$ is bounded.

(iii)From the dynamics (\ref{closedloop-sat}) and (ii), it can be shown
that $\frac{d}{dt}(-B^{T}\frac{\partial H}{\partial x}(x)-x_c)$ is
bounded. Combining the facts that $V(x,x_c)$ is bounded from below with
$\dot{V}\leqslant0$, we have that $\lim_{t\rightarrow\infty}
\dot{V}(x(t),x_c(t))=0.$

Indeed, suppose $\dot{V}(x(t),x_c(t))$ does not converge to zero. In
other words,
there exists a real $\delta>0$ and a sequence $\{t_k\}$, satisfying
$\lim_{k\rightarrow\infty}t_k=+\infty$, such that
$\dot{V}(x(t_k),x_c(t_k))<-\delta.$
Since $\frac{d}{dt}(-B^T\frac{\partial H}{\partial x}(x)-x_c)$ is
bounded, 
then for each $k=1,2,\ldots,$ there exists a time interval $I_k$ and an
$\epsilon>0$ such that $|I_k|>\epsilon, t_k\in I_k,$ and $\forall t\in I_k,
\dot{V}(x(t),x_c(t))<-\frac{\delta}{2}$. This implies that 
\begin{equation}
\lim_{t\rightarrow \infty}
V(x(t),x_c(t))= -\infty, \nonumber
\end{equation}
which is contradicted by (i). In conclusion, $ \lim_{t\rightarrow \infty}
\dot{V}(x(t),x_c(t))=0.$
\end{proof}

We obtain our main theorem.

\begin{theorem}\label{main}
 Consider the dynamical system (\ref{closedloop-sat}) defined on a weakly
connected directed graph
with compatible constraints $[u^-, u^+]$. Then the trajectories will converge
into 
\begin{equation}
 \mathcal{E}_{\mathrm{tot}} = \{ (x,x_c) \mid \frac{\partial H}{\partial x}(x) =
\alpha \mathds{1}_n, \, B\sat(-x_c\,;u^-,u^+) = 0 \}.
\end{equation}
if and only if the network satisfies the interior point condition.
\end{theorem}

\begin{proof}
 {\it Sufficiency:}
Suppose the network satisfies interior point condition with vector $z$ which
has representation as (\ref{z representation}). Consider the following function
\begin{equation}
V(x,x_c)=\mathds{1}^{T} S\big(-B^{T}\frac{\partial H}{\partial
x}(x)-x_c\,;u^-,u^+\big)+H(x)
\end{equation}
as Lyapunov function. Notice that $\dot{V}=-\dot{x}^T\frac{\partial ^2
H}{\partial x^2}\dot{x}\leq0.$

Using Lemma \ref{lower bound} and LaSalle's principle, it follows that
$(x(t),x_c(t))$ converges to the largest invariant set $\mathcal{I}$ contained
in $\{(x,x_c) \mid  \dot{V}=0\,\}.$
If a solution $(x(t),x_c(t))\in \mathcal{I}$, then $x$ is a constant
vector, denoted as $\nu$. Furthermore, $\mathcal{I}$ is given as
\begin{equation}
\begin{aligned}
 &\mathcal{I}=\{(\nu,x_c)\mid x_c=B^T\frac{\partial H}{\partial
x}(\nu)t+x_c(0),\\
&B\sat\big(-B^T\frac{\partial H}{\partial
x}(\nu)-B^T\frac{\partial H}{\partial
x}(\nu) t-x_c(0)\,;u^-,u^+\big)=0, \\
& \forall t\geq 0\}.
\end{aligned}
\end{equation}

We will prove $B^T\frac{\partial H}{\partial x}(\nu)=0$ by contradiction. 
Suppose we choose $(\nu,x_c(0))\in\mathcal{I}$, according to the definition of
invariant set, $V\big(\nu,B^{T}\frac{\partial H}{\partial x}(\nu)t+x_c(0)\big)$
is a constant along the trajectory. Furthermore, 
\begin{equation}
V_1(t):=\mathds{1}^{T} S\big(-B^{T}\frac{\partial H}{\partial
x}(\nu)-B^{T}\frac{\partial H}{\partial x}(\nu)t-x_c(0)\,;u^-,u^+\big)
\end{equation}
will be a constant. Notice that 
\begin{equation}
\begin{aligned}
 \dot{V}_1 = & -\sat^T\big(-B^{T}\frac{\partial H}{\partial
x}(\nu)-B^{T}\frac{\partial H}{\partial
x}(\nu)t-x_c(0),u^-,u^+\big) \\
& B^T\frac{\partial H}{\partial x}(\nu)
\end{aligned}
\end{equation}
We can assume that $\dot{V}_1(0)=0$.

Suppose if on an edge $e_i$, $B_i^T\frac{\partial H}{\partial
x}(\nu)>0$, the for large enough $t$
\begin{equation}
 \begin{aligned}
  u^-_i & = \sat\big(-B_i^{T}\frac{\partial H}{\partial
x}(\nu)-B_i^{T}\frac{\partial H}{\partial
x}(\nu)t-x_{c_i}(0),u^-_i,u^+_i\big) \\
& \leq \sat\big(-B_i^{T}\frac{\partial H}{\partial
x}(\nu)-x_{c_i}(0),u^-_i,u^+_i\big),
 \end{aligned}
\end{equation}
while if $B_i^T\frac{\partial H}{\partial
x}(\nu)<0$, the for large enough $t$
\begin{equation}
 \begin{aligned}
  u^+_i & = \sat\big(-B_i^{T}\frac{\partial H}{\partial
x}(\nu)-B_i^{T}\frac{\partial H}{\partial
x}(\nu)t-x_{c_i}(0),u^-_i,u^+_i\big) \\
& \geq \sat\big(-B_i^{T}\frac{\partial H}{\partial
x}(\nu)-x_{c_i}(0),u^-_i,u^+_i\big).
 \end{aligned}
\end{equation}

Furthermore, if the edge $e_i\in\mathcal{E}_2$, the above two inequalities hold
strictly, which implies that $\dot{V}_1(t)>0$ for large enough $t$. This is a
contradiction. So along all the edges of $\mathcal{E}_2$, $B_i^T\frac{\partial
H}{\partial x}(\nu)=0$.
Since $(\mathcal{V},\mathcal{E}_2)$
contains a spanning tree, $B^T\frac{\partial H}{\partial x}(\nu)=0$

{\it Necessity}
First of all, if there does not exist $z$ such that $B\sat(z; u^-, u^+)=0$, the
system (\ref{closedloop-sat}) is unstable.
Suppose now the network does not satisfy interior point condition, i.e there
exist a vector $z$ such that 
\begin{equation}\label{nece}
 Bz=0, z\in[u^-,u^+].
\end{equation}
however for any $z$ such that (\ref{nece}) holds, the $\mathcal{E}_2(z;u^-,u^+)$
does not contain a spanning tree.

For this case, we will show that the dynamical system (\ref{closedloop-sat})
will form a clustering by setting suitable initial
condition $(x(0),x_c(0))$ with $B^T\frac{\partial H}{\partial x}(x(0))\neq0$
and $\dot{x}(t)=0, t\geq0.$

Since $\mathcal{E}_2(z;u^-, u^+)\cup\mathcal{E}_3(z;u^-, u^+)=\mathcal{E},
\mathcal{E}_2(z;u^-, u^+)\cap\mathcal{E}_3(z;u^-, u^+)=\emptyset$ and
$\mathcal{E}_2(z;u^-, u^+)$ does not contain
a spanning tree, then $\mathcal{E}_3(z;u^-, u^+)$ contains a cut set. Suppose
the graph $\mathcal{G}_2(z)$ is not weakly connected and has $k$ weakly
connected components, denoted as $\mathcal{G}_2^1(z),\cdots,$
$\mathcal{G}_2^k(z)$, we can introduce a reduced graph $\tilde{\mathcal{G}}(z)$
which has $k$ vertices that each of them represents a component of
$\mathcal{G}_2(z)$,
$\mathcal{E}(\tilde{\mathcal{G}}(z))\subseteq\mathcal{E}_3(z;u^-, u^+)$ is the
set of edges connecting the components of $\mathcal{G}_2(z)$. With a slight
abuse of notation, we denote the vertices of $\tilde{\mathcal{G}}(z)$ as
$\mathcal{G}_2^1(z),\cdots,\mathcal{G}_2^k(z)$ too. This reduction is shown in
Figure \ref{simplified}.

Next we will conduct the following algorithm on $\tilde{\mathcal{G}}(z)$.

Algorithm: Initialization, find any $z^0\in\mathbb{R}^m_{\geq0}$ such that
(\ref{nece}) holds. Let $z^k$ denote the value of $z$ from the previous
iteration. For $z^k$, check if $\tilde{\mathcal{G}}(z^k)$ satisfies case 1 or
case 2 which are given below. If does, we will derive a new $z^{k+1}$ satisfying
(\ref{nece}) and repeat this step for $\tilde{\mathcal{G}}(z^{k+1})$; if not,
the algorithm stops. Notice that $\tilde{\mathcal{G}}(z^{k+1})$ has fewer
vertices than $\tilde{\mathcal{G}}(z^k)$.

Case 1: Consider the subgraph of $\tilde{\mathcal{G}}(z^k)$ given
as in Figure \ref{combined}(left), where
$\mathcal{G}^i_2(z^k),\mathcal{G}^j_2(z^k)$ are two
nodes in $\tilde{\mathcal{G}}(z^k)$ which are connected by edges $e_{l_1}$ and
$e_{l_2}$. Suppose $z^k$ on $e_{l_1}$ and
$e_{l_2}$ reach different bounds, for instance, upper bound on $e_{l_1}$ 
and lower bound on $e_{l_2}$, then we can modify $z^k$ to $z^{k+1}$
satisfying (\ref{nece}) such that $z^{k+1}$ belongs to interior of constraint
intervals on $e_{l_1}$ and $e_{l_2}$. Besides $\mathcal{G}^i_2(z^k)$ and
$\mathcal{G}^j_2(z^k)$ will merge into one node in
$\tilde{\mathcal{G}}(z^{k+1})$. Indeed, in this case there exist a closed path
of $\mathcal{G}$ composed by the edges in
$\mathcal{G}^i_2(z^k),\mathcal{G}^j_2(z^k)$, $e_{l_2}$ and reversed $e_{l_1}$,
with incidence vector denoted as
$w$, such that for small enough $\epsilon>0$,
$z^{k+1}_{l_i}\in int[u^-_{l_i},u^+_{l_i}], i=1,2$ where
$z^{k+1}=z^k+\epsilon w$. 

Case 2: Suppose along any positive circuit in $\tilde{\mathcal{G}}(z^k)$, $z^k$
reaches the same bounds, i.e.
upper bounds simultaneously or lower bounds. An example is given as in Figure
\ref{combined}(right), where $z^k$ on $e_{l_1}, e_{l_2},e_{l_3}$ reaches lower
bounds at the same time. Then similar to the previous case, there exists a
closed path of $\mathcal{G}$ composed by the edges in
$\mathcal{G}^h_2(z^k),\mathcal{G}^i_2(z^k),\mathcal{G}^j_2(z^k)$ and
$e_{l_i},i=1,2,3$ with the incidence vector denoted as $w$, such that for
small enough $\epsilon>0$, $z^{k+1}_{l_i}\in int[u^-_{l_i},u^+_{l_i}],
i=1,2,3$ where $z^{k+1}=z^k+\epsilon w$. For the case when $z^k$ on $e_{l_1},
e_{l_2},e_{l_3}$ reaches upper bounds, we could use the reversed closed path
with incidence vector $-w$.

Since there exist only a finite number of vertices of $\mathcal{G}$, the
algorithm will stop after finite steps. Let us denote the final value of $z$
as $z^*$. If the network satisfies the interior point condition, then
$\tilde{\mathcal{G}}(z^*)$ is a trivial graph with only one vertex. If not, the
graph $\tilde{\mathcal{G}}(z^*)$ satisfies: first, $z^*$ on the edges of
$\tilde{\mathcal{G}}(z^*)$ with the same starting and ending nodes reaches the
same bounds; second, along any positive circuit in $\tilde{\mathcal{G}}(z^*)$,
$z^*$ reaches upper and lower bounds simultaneously.

Finally we can set the suitable initial condition of system
(\ref{closedloop-sat}) on $\mathcal{G}$ such that $\dot{x}=0,
B^T\frac{\partial H}{\partial x}\neq0, \forall t>0$. Based on $z^*$, we can set
$x_c(0)=-z^*$. For $\frac{\partial
H}{\partial x}$, we can assign it the same value in each weakly
connected component of $\mathcal{G}_2(z^*)$. In fact, we can set it on
$\mathcal{G}^i_2(z^*)$ larger than it on $\mathcal{G}^j_2(z^*)$ if there is an
edge from $\mathcal{G}^j_2(z^*)$ to $\mathcal{G}^i_2(z^*)$ on which $z^*$
reaches lower bound. Similarly, $\frac{\partial
H}{\partial x}$ on $\mathcal{G}^i_2(z^*)$ is assigned to be
smaller than it on $\mathcal{G}^j_2(z^*)$ if $z^*$ reaches upper bound on the
edges from $\mathcal{G}^j_2(z^*)$ to $\mathcal{G}^i_2(z^*)$.
We can verify that $\dot{x}(t)=0$, but $B^T\frac{\partial H}{\partial
x}(x(t))\neq0, \forall t>0$.
\end{proof}

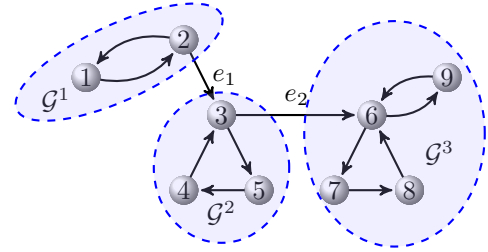
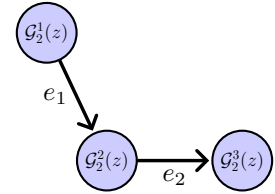
\begin{figure}
\centering
\subfigure[The whole network $\mathcal{G}$ with its weakly connected
components]{
\begin{tikzpicture}[->,>=stealth',shorten >=1pt,auto,node distance=3cm,
 thick,main
node/.style={circle,fill=blue!20,draw,font=\sffamily\Large\bfseries}]

\useasboundingbox (0,0) rectangle (5cm,4cm);

\tikzstyle{VertexStyle} = [shading         = ball,
                           ball color      = white!100!white,
                           minimum size = 20pt,%
                           inner sep       = 1pt,]
\Vertex[style={minimum
size=0.2cm,shape=circle},LabelOut=false,L=\hbox{$1$},x=0.5cm,y=2.5cm]{v1}
\Vertex[style={minimum
size=0.2cm,shape=circle},LabelOut=false,L=\hbox{$2$},x=1.8cm,y=3cm]{v2}
\Vertex[style={minimum
size=0.2cm,shape=circle},LabelOut=false,L=\hbox{$3$},x=2.3cm,y=2cm]{v3}
\Vertex[style={minimum
size=0.2cm,shape=circle},LabelOut=false,L=\hbox{$4$},x=1.8cm,y=1cm]{v4}
\Vertex[style={minimum
size=0.2cm,shape=circle},LabelOut=false,L=\hbox{$5$},x=2.8cm,y=1cm]{v5}
\Vertex[style={minimum
size=0.2cm,shape=circle},LabelOut=false,L=\hbox{$6$},x=4.3cm,y=2cm]{v6}
\Vertex[style={minimum
size=0.2cm,shape=circle},LabelOut=false,L=\hbox{$7$},x=3.8cm,y=1cm]{v7}
\Vertex[style={minimum
size=0.2cm,shape=circle},LabelOut=false,L=\hbox{$8$},x=4.8cm,y=1cm]{v8}
\Vertex[style={minimum
size=0.2cm,shape=circle},LabelOut=false,L=\hbox{$9$},x=5.3cm,y=2.5cm]{v9}
\draw
(v1) edge[bend right](v2)
(v2) edge[bend right](v1)
(v2) edge node[right]{$e_1$} (v3)
(v4) edge  (v3)
(v3) edge  (v5)
(v3) edge node[above]{$e_2$} (v6)
(v6) edge  (v7)
(v8) edge  (v6)
(v6) edge [bend right] (v9)
(v9) edge [bend right] (v6)
(v5) edge  (v4)
(v7) edge  (v8);

\node at (0.1,2.2){\text{$\mathcal{G}^1$}};
\node at (2.3,0.7){\text{$\mathcal{G}^2$}};
\node at (5.2,1.5){\text{$\mathcal{G}^3$}};

\draw[rotate around={25:(1.15cm, 2.25cm)},dashed,fill=blue!30,fill
opacity=0.2, dashed, draw=blue]
(1.15cm, 2.75cm) ellipse (1.5cm and 0.5cm);
\draw[dashed,fill=blue!30,fill
opacity=0.2, dashed, draw=blue]
(2.3cm, 1.3cm) ellipse (0.9cm and 1cm);
\draw[dashed,fill=blue!30,fill
opacity=0.2, dashed, draw=blue]
(4.6cm, 1.75cm) ellipse (1.2cm and 1.5cm);

       
\end{tikzpicture}}
 \hspace{1in}
\subfigure[Simplified graph
$\tilde{\mathcal{G}}(z)$]{\begin{tikzpicture}[->,>=stealth',shorten
>=1pt,auto,node distance=3cm,
  thick,main
node/.style={circle,scale=0.5,fill=blue!20,draw,
font=\sffamily\Large\bfseries}]
\tikzstyle{EdgeStyle}    = [thin,double= black,
                            double distance = 0.5pt]
\useasboundingbox (0,0) rectangle (3.5cm,3cm);

\node[main node](1) at (0.4cm, 2cm){$\mathcal{G}_2^1(z)$};
\node[main node](2) at (1.2cm, 0.3cm){$\mathcal{G}_2^2(z)$};
\node[main node](3) at (3cm, 0.3cm){$\mathcal{G}_2^3(z)$};
\draw
 (1) edge[->,>=angle 90,thin,double= black,double distance = 0.5pt]
node[left]{$e_1$} (2)
(2) edge[->,>=angle 90,thin,double= black,double distance = 0.5pt]
node[below]{$e_2$} (3);
\end{tikzpicture}
}
\caption{For a given $z$, the pony-shape network, given as (a), falls into three
weakly connected components after deleting $\mathcal{E}_3$ where
$e_1,e_2\in\mathcal{E}_3$. By denoting each weakly connected component as a
node, we get the simplified graph $\tilde{\mathcal{G}}(z)$, given as in
(b). In this case $\mathcal{G}_2^i(z),i=1,2,3$ represent
either a node in $\tilde{\mathcal{G}}(z)$ or a weakly component of 
$\mathcal{G}_2(z)$.}\label{simplified}
\end{figure}

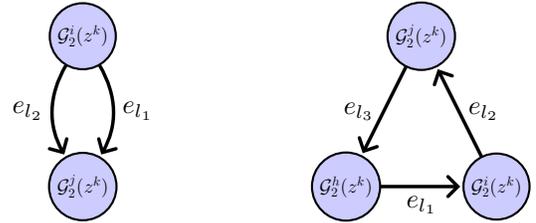
\begin{figure}
\centering
\begin{tikzpicture}[->,>=stealth',shorten >=1pt,auto,node distance=3cm,
  thick,main
node/.style={circle,scale=0.5, fill=blue!20,draw,font=\sffamily\Large\bfseries}]
\tikzstyle{EdgeStyle}    = [thin,double= black,
                            double distance = 0.5pt]
\useasboundingbox (0,0) rectangle (7cm,2.5cm);

\node[main node](1) at (1, 2.2){$\mathcal{G}^i_2(z^k)$};
\node[main node](2) at (1, 0.2){$\mathcal{G}^j_2(z^k)$};
\node[main node](3) at (4.5, 0.2){$\mathcal{G}^h_2(z^k)$};
\node[main node](4) at (6.5, 0.2){$\mathcal{G}^i_2(z^k)$};
\node[main node](5) at (5.5, 2.2){$\mathcal{G}^j_2(z^k)$};
\draw
 (3) edge[->,>=angle 90,thin,double= black,double distance = 0.5pt]
node[below]{$e_{l_1}$} (4)
(4) edge[->,>=angle 90,thin,double= black,double distance = 0.5pt]
node[right]{$e_{l_2}$} (5)
(5) edge[->,>=angle 90,thin,double= black,double
distance = 0.5pt]node[left]{$e_{l_3}$} (3)
(1) edge[->,>=angle 90,bend left,thin,double= black,double distance = 0.5pt]
node[right]{$e_{l_1}$} (2)
(1) edge[->,>=angle 90,bend right,thin,double= black,double distance = 0.5pt]
node[left]{$e_{l_2}$} (2);
\end{tikzpicture}

\caption{Two cases which are studied in proof of Theorem \ref{main}.}
\label{combined}
\end{figure}

\begin{remark}
 For any final value $z^*$ from the previous algorithm, notice that in any
positive circuit of $\mathcal{G}$, there is either no edge of
$\tilde{\mathcal{G}}(z^*)$ or at least two. 

First, we will show that for a weakly connected network satisfying Interior
Point Condition, the algorithm will not end up with a graph
$\tilde{\mathcal{G}}(z^*)$ with more than one node for any initial condition
$z^0$. Indeed, suppose $\tilde{\mathcal{G}}(z^*)$ has more than one node, by
Lemma \ref{positive circuit}, $z-z^*$ can be represented by a linear combination
of positive circuits where $z$ is a vector from Interior Point Condition.
However, since $\tilde{\mathcal{G}}(z^*)$ does not satisfy Case 2, for any
linear combination of positive circuits, denoted as $w$, $z^*+w\notin[u^-,u^+]$.
This is contradicted to $z\in[u^-,u^+]$.

Next we will show that for a weakly connected network the algorithm will not
produce different $\tilde{\mathcal{G}}(z^*_1)$ and $\tilde{\mathcal{G}}(z^*_2)$
by using different initial conditions $z^0_1,z^0_2$. Indeed suppose there exist
two nodes $x_i,x_j$ which belong to the same node of
$\tilde{\mathcal{G}}(z^*_1)$ but two different nodes of
$\tilde{\mathcal{G}}(z^*_2)$, denoted as $\mathcal{G}_2^i(z^*_2)$ and
$\mathcal{G}_2^j(z^*_2)$. By Lemma \ref{positive circuit}, $z^*_1-z^*_2$ can be
represented as linear combination of positive circuits. However for any
linear combination of positive circuits, denoted as $w$,
$z^*_2+w\notin[u^-,u^+]$. This is contradicted to $z^*_1\in[u^-,u^+]$.

In conclusion, Interior Point Condition is a property of the network which does
not depend the initial condition of the algorithm.
\end{remark}

\section{Controlling to arbitrary steady states in the admissible set}

Notice that the first equation in system (\ref{system1}) implies that
\begin{equation}
 \mathds{1}^T\dot{x}=0
\end{equation}
if the disturbance satisfies the matching condition (\ref{matching}).

In this section, instead of driving the state to consensus, we will make
the state converge to any desirable point $x^*$ in the admissible set, which is
defined as
\begin{equation}
 \mathcal{A}=\{x \mid \mathds{1}^Tx=\mathds{1}^Tx(0)\}.
\end{equation}
We will achieve this by modifying our controller
(\ref{unconstrained-controller})

In \cite{Blanchini00}, the authors construct a discontinuous controller with
the information of desirable point to complete this task. Here we consider the
following controller (with gain matrix $R=I$)
\begin{equation}\label{controller-arbitrary}
 \begin{aligned}
  \dot{x}_c & = B^T\frac{\partial H}{\partial x}(x-x^*)\\
u & = -B^T\frac{\partial H}{\partial x}(x-x^*)-\frac{\partial H_c}{\partial
x_c}
 \end{aligned}
\end{equation}
which can measure the {\it difference} between the current states and the
desirable ones. If the input is not saturated, then the closed-loop system is
given as
 \begin{equation}\label{closedloop-arbitrary}
 \begin{bmatrix} \dot{x} \\[2mm] \dot{x}_c \end{bmatrix} =
\begin{bmatrix} -BB^T & -B \\[2mm] B^T & 0 \end{bmatrix}
\begin{bmatrix} \frac{\partial H}{\partial x}(x-x^*) \\[2mm] \frac{\partial
H_c}{\partial x_c}(x_c) \end{bmatrix} +
\begin{bmatrix} E \\[2mm] 0 \end{bmatrix} d,
\end{equation}

\begin{theorem}
 Suppose $H(x)$ is strictly convex and has its minimum at the origin. Consider
a desirable state $x^*$ from admissible set $\mathcal{A}$. Assume the
disturbance satisfies the matching condition (\ref{matching}). Then the
trajectories of system
(\ref{closedloop-arbitrary}) will converge to 
\begin{equation}
 \mathcal{E}_{tot}=\{(x^*, \bar{x}_c) \mid B\frac{\partial
H_c}{\partial \bar{x}_c}(x_c)=E\bar{d} \}
\end{equation}
if and only if the graph is weakly connected.
\end{theorem}

\begin{proof}

 {\it Sufficiency}:
Consider the modified Hamiltonian 
\begin{equation}
 H^*(x,x_c)=H(x-x*)+H_c(x_c)-H_c(\bar{x}_c)-\frac{\partial^T
H_c}{\partial \bar{x}_c}(x_c-\bar{x}_c)
\end{equation}
as Lyapunov function. Then we have 
\begin{equation}
  \frac{d H^*}{dt}  = -\frac{\partial^T
H}{\partial x}(x-x^*)BB^T\frac{\partial H}{\partial x}(x-x^*) \leq 0
\end{equation}
By LaSalle's principle, the trajectories will converge into the largest
invariant set, denoted as $\mathcal{I}$, in $\{(x,x_c)\mid B^T\frac{\partial
H}{\partial x}(x-x^*)=0\}$. By the fact that the network is weakly connected,
$\frac{\partial
H}{\partial x}(x-x^*)=\alpha\mathds{1}$ for some $\alpha\in\mathbb{R}$. Since
$H$ is strictly convex, we can write $x-x^* = \frac{\partial
H ^{-1}}{\partial x}(\alpha\mathds{1})$. Furthermore, because $H$ gets minimum
at origin and $\mathds{1}^T(x-x^*)=0$, so $\alpha=0$ and $x=x^*$. The rest of
proof follows the standard way.

{\it Necessary}: If the network is not weakly connected, the controller
(\ref{controller-arbitrary}) can only drive the states into the admissible set
corresponding to each weakly connected components. Here end the proof
\end{proof}

By denoting $x-x^*$ as $\tilde{x}$, we can write the system
(\ref{closedloop-arbitrary}) as 
\begin{equation}
 \begin{bmatrix} \dot{\tilde{x}} \\[2mm] \dot{x}_c \end{bmatrix} =
\begin{bmatrix} -BB^T & -B \\[2mm] B^T & 0 \end{bmatrix}
\begin{bmatrix} \frac{\partial H}{\partial \tilde{x}}(\tilde{x}) \\[2mm]
\frac{\partial
H_c}{\partial x_c}(x_c) \end{bmatrix} +
\begin{bmatrix} E \\[2mm] 0 \end{bmatrix} d.
\end{equation}
After absorbing the disturbance into constraint intervals as
(\ref{disturbance_in_const}), the corresponding saturated case can be written as
\begin{equation}\label{closeloop-sat-arbitrary}
 \begin{aligned}
\dot{\tilde{x}} & =  B\sat\big(-B^T\frac{\partial H}{\partial
\tilde{x}}(\tilde{x})-x_c\,;u^-,u^+\big),
\\[2mm]
\dot{x}_c & =  B^T\frac{\partial H}{\partial \tilde{x}}(\tilde{x}).
\end{aligned}
\end{equation}

By the fact that $\mathds{1}^T\tilde{x}=0$, and $H$ is strictly convex and
reach minimum at origin, we can show that $\frac{\partial H}{\partial
\tilde{x}}(\tilde{x})\rightarrow \alpha\mathds{1}\Rightarrow \alpha =0$ and
$\tilde{x}\rightarrow 0.$ This is a direct application of Theorem \ref{main}.

\begin{corollary}
 Suppose the Hamiltonian $H$ is strictly convex and reach minimum at
origin. For any $x^*\in\mathcal{A}$, consider
dynamical system (\ref{closeloop-sat-arbitrary}) defined on a directed graph
with compatible constraints $[u^-, u^+]$, then the trajectories will converge
into 
\begin{equation}
 \mathcal{E}_{\mathrm{tot}} = \{ (0,x_c) \mid  B\sat(-x_c\,;u^-,u^+) = 0 \}.
\end{equation}
if and only if the network satisfies interior point condition.
\end{corollary}

\section{CONCLUSIONS}
We have discussed a basic model of dynamical distribution networks where the flows through the edges are generated by distributed PI controllers. The main part of this
paper focuses on the case where flow constraints are present. A key ingredient in this analysis is the construction of a $C^1$ Lyapunov function. After making the orientation and constraint interval compatible, we derived a sufficient and necessary condition, which depends on the graphic structure of the network and constraint intervals, for asymptotic load balancing with any given network and constraints. By modified PI controller, the states on nodes can be driven to any desirable state of admissible set.

An obvious open problem is how to put constraint
on storage variables on vertices. This is currently under investigation. Many other questions can be addressed in this framework. For example, what is happening if
the in/outflows are not assumed to be constant, but are e.g. periodic
functions of time; see already \cite{depersis}.

\addtolength{\textheight}{-12cm}   



\bibliographystyle{IEEEtran} 
\bibliography{ifacconf}

\begin{thebibliography}{10}
\providecommand{\url}[1]{#1}
\csname url@rmstyle\endcsname
\providecommand{\newblock}{\relax}
\providecommand{\bibinfo}[2]{#2}
\providecommand\BIBentrySTDinterwordspacing{\spaceskip=0pt\relax}
\providecommand\BIBentryALTinterwordstretchfactor{4}
\providecommand\BIBentryALTinterwordspacing{\spaceskip=\fontdimen2\font plus
\BIBentryALTinterwordstretchfactor\fontdimen3\font minus
  \fontdimen4\font\relax}
\providecommand\BIBforeignlanguage[2]{{%
\expandafter\ifx\csname l@#1\endcsname\relax
\typeout{** WARNING: IEEEtran.bst: No hyphenation pattern has been}%
\typeout{** loaded for the language `#1'. Using the pattern for}%
\typeout{** the default language instead.}%
\else
\language=\csname l@#1\endcsname
\fi
#2}}

\bibitem{wei2013}
{J. Wei and A.J. van der Schaft}, ``Load balancing of dynamical distribution
  networks with flow constraints and unknown in/outflows,'' \emph{Systems \&
  Control Letters}, vol. 62(11), pp. 1001--1008, 2013.

\bibitem{depersis}
C.~D. Persis, ``Balancing time-varying demand-supply in distribution networks:
  an internal model approach,'' \emph{European Control Conference
  (ECC),Z\"{u}rich, Switzerland.}, 2013.

\bibitem{Blanchini00}
F.~Blanchini, S.Miani, and W.Ukovich, ``Control of production-distribution
  systems with unknown inputs and system failures,'' \emph{Automatic Control,
  IEEE Transactions on}, vol.~45, no.~6, pp. 1072--1081, 2000.

\bibitem{Ren08}
W.~Ren, ``On consensus algorithms for double-integrator dynamics,''
  \emph{Automatic Control, IEEE Transactions on}, vol.~53, no.~6, pp.
  1503--1509, 2008.

\bibitem{jaya1}
{B. Jayawardhana and R. Ortega and E. Garc\'{\i}a-Canseco and F. Casta\~{n}os},
  ``Passivity of nonlinear incremental systems: Application to {PI}
  stabilization of nonlinear {RLC} circuits,'' \emph{Systems \& Control
  Letters}, vol.~56, pp. 618--622, 2007.

\bibitem{Bollobas98}
{B. Bollobas}, \emph{Modern Graph Theory}, ser. Graduate Texts in
  Mathematics.\hskip 1em plus 0.5em minus 0.4em\relax New York: Springer, 1998,
  vol. 184.

\bibitem{schaftNECSYS10}
{A.J. van der Schaft and B.M. Maschke}, ``Port-{H}amiltonian dynamics on
  graphs: Consensus and coordination control algorithms,'' \emph{Proc. 2nd IFAC
  Workshop on Distributed Estimation and Control in Networked Systems, Annecy,
  France}, pp. 175--178, 2010.

\bibitem{schaftCDC08}
------, ``Conservation laws on higher-dimensional networks,'' \emph{Proc. 47th
  IEEE Conf. on Decision and Control}, pp. 799--804, 2008.

\bibitem{schaftSIAM}
{A.J. van~der Schaft and B.M. Maschke}, ``Port-{H}amiltonian systems on
  graphs,'' \emph{SIAM J. Control and Optimization}, vol. 51(2), pp. 906--937,
  2013.

\bibitem{vanderschaftmaschkearchive}
{A.J van~der Schaft and B.M. Maschke}, ``The {H}amiltonian formulation of
  energy conserving physical systems with external ports,'' \emph{Archiv
  f\"{u}r Elektronik und \"{U}bertragungstechnik}, vol.~49, pp. 362--371, 1995.

\bibitem{vanderschaftbook}
{A.J. van der Schaft}, \emph{$L_2$-Gain and Passivity Techniques in Nonlinear
  Control}, ser. Lecture Notes in Control and Information Sciences.\hskip 1em
  plus 0.5em minus 0.4em\relax Berlin: Springer-Verlag, 1996, vol. 218, 2nd
  edition, Springer, London, 2000.

\bibitem{gatermann2005}
K.~Gatermann and M.~Wolfrum, ``Bernstein's second theorem and viro's method for
  sparse polynomial systems in chemistry,'' \emph{Advances in Applied
  Mathematics}, vol.~34, no.~2, pp. 252 -- 294, 2005.

\bibitem{gatermann2002}
K.~Gatermann, ``Chemical reactions stoichiometric network analysis,''
  \emph{vorlesung im sommersemester 2002 an der freien universit\"{a}t berlin},
  2002.

\bibitem{othmer1981}
H.~Othmer, ``A graph-theoretic analysis of chemical reaction networks,''
  \emph{Lecture Notes, Rutgers University}, 1981.

\end{thebibliography}

\end{document}